\documentclass[12pt, a4paper, reqno]{amsart}

\usepackage{hyperref}
\usepackage{amsmath, amsthm, amssymb, amsrefs,mathrsfs}
\usepackage[reftex]{theoremref}
\usepackage{esint, bigints}
\usepackage{enumitem}

\theoremstyle{definition}
\newtheorem{definition}{Definition}[section]
\newtheorem{remark}[definition]{Remark}
\newtheorem{example}[definition]{Example}

\theoremstyle{plain}
\newtheorem{proposition}[definition]{Proposition}
\newtheorem{theorem}[definition]{Theorem}
\newtheorem{lemma}[definition]{Lemma}
\newtheorem{corollary}[definition]{Corollary}

\newcommand{\ee}{\ensuremath{\,\mathrm{e}}}

\newcommand{\dd}{\ensuremath{\,\mathrm{d}}}
\newcommand{\R}{\ensuremath{\mathbb{R}}}

\newcommand{\loc}{\ensuremath{\mathrm{loc}}}

\DeclareMathOperator{\Ch}{Ch}
\DeclareMathOperator{\dom}{Dom}
\DeclareMathOperator{\LIP}{Lip}
\DeclareMathOperator{\lip}{lip}

\DeclareMathOperator{\supp}{supp}

\DeclareMathOperator{\CBB}{CBB}
\DeclareMathOperator{\Dyn}{Dyn}
\DeclareMathOperator{\CAP}{Cap}

\DeclareMathOperator{\BICB}{BICB}
\DeclareMathOperator{\RCD}{RCD}
\DeclareMathOperator{\Cone}{Cone}

\title{Analysis on surfaces with locally bounded integral curvature}

\author{Sebastian Boldt}
\author{Batu Güneysu}
\author{Maxime Marot}

\address{Mathematisches Seminar, Heinrich-Hecht-Platz 6, 24118 Kiel, Germany}
\email{sebastian.boldt@mathematik.tu-chemnitz.de}
\email{gueneysu@math.uni-kiel.de}
\email{marot@math.uni-kiel.de}

\begin{document}

\begin{abstract}
We prove several analytic results on (possibly noncompact) complete singular surfaces having locally bounded integral curvature (in short: BIC surfaces). Regarding these as metric measure spaces with the $2$-dimensional Hausdorff measure, we show that these are infinitesimally Hilbertian, locally doubling and satisfy a local Poincaré inequality. In particular, this entails the existence of a jointly Hölder continuous heat kernel for the Cheeger Laplacian. Assuming that the negative part of the curvature measure of a BIC surface satisfies a Dynkin-type condition, we show that the surface is bi-Lipschitz equivalent to a BIC surface with a lower bounded curvature measure, entailing global variants of the aforementioned results.
\end{abstract}

\maketitle

\tableofcontents

\section{Introduction}
The class of \emph{surfaces with locally bounded integral curvature} (BIC surfaces), going back to Alexandrov and his school \cite{AleksandrovIntrinsicGeometry} in the 1960s and further developed by Reshetnyak \cite{ReshetnyakTwoDimManifolds} and Huber \cite{huberPotentialtheoretischenAspektAlexandrowschen1960} (see also \cites{fillastreSubharmonic2023,troyanov2022alexandrovssurfacesboundedintegral} for an overview), is arguably the largest class of metrically singular surfaces where a natural notion of curvature can be formulated in a way that connects the geometry and the topology of the space: roughly speaking, these are length spaces, homeomorphic to smooth surfaces, whose curvature excess over triangles is only assumed to be locally bounded, with no sign condition imposed. These assumptions, being in big contrast to the more restrictive curvature-bounded-below (CBB) condition of Alexandrov, allow one to model surfaces with conical singularities, highly oscillating curvature, and also cusps. As indicated above, every BIC surface carries a natural curvature measure, which turns out to be a locally bounded signed measure in general, and which allows one to formulate a Gauss-Bonnet theorem in the compact case \cite[Theorem 5 p.192]{AleksandrovIntrinsicGeometry}. Moreover, by the Reshetnyak-Huber uniformization \thref{thm:ReschetnyakHuber} (see \cites{reshetnyakIsothermalI2023, reshetnyakIsothermalII2023, huberPotentialtheoretischenAspektAlexandrowschen1960, chenUniformConvergenceMetrics2025}), every BIC surface can be represented, locally and even globally away from cusps, as a smooth Riemannian surface conformally deformed by a nonsmooth subharmonic function.\vspace{1mm}

Despite the naturality of this class, a systematic study of analysis on BIC surfaces from the viewpoint of Cheeger energies and Sobolev spaces on metric measure spaces appears to be missing from the literature. At the level of global results, this is not surprising at all: On the one hand, if $(S,d)$ is a BIC surface with $\omega$ its curvature measure and $\mu$ its two-dimensional Hausdorff measure, then given $\kappa\in\R$, the natural way of saying that $(S,d)$ has curvature bounded from below by $\kappa$ is to assume $\omega\geq \kappa \mu$, the so-called $\BICB(\kappa)$ condition. Only recently it has been shown by the third-named author in \cite{marot2025notionscurvaturesingularsurfaces} that this condition is -- under weak assumptions -- equivalent to the $\mathrm{CBB}(\kappa)$ condition, which by Petrunin's work \cite{petruninAlexandrovMeetsLottVillaniSturm2011} shows that $(S,d,\mu)$ is an $\RCD(\kappa,2)$ space, noting that on $\RCD(\kappa,N)$ spaces all sorts of global analytic results are available (e.g. locally uniform doubling \cite{sturm2}, local Poincaré inequality \cite{rajalaLocalPoincare2012}, global Gaussian heat kernel bounds \cite{jiangHeatKernelBound2016}). On the other hand, if $\omega^-$ is sufficiently ill-behaved, one escapes the $\RCD(\kappa,2)$ framework entirely and no ready-to-use results are available at all. In fact, the only analytic results we are aware of in this setting are \cite{adamowiczIsoperimetricInequalitiesGeometry2021}, where Adamowicz-Veronelli treat the Dirichlet problem on annuli, and \cite{deruelleKahlerRicciFlows2025}, where Deruelle-Guedj-Guenancia-Zeriahi prove that one can start a Ricci flow from a BIC surface.\vspace{1mm}

The purpose of this article is to fill this gap by providing the foundations of a geometric analysis framework on BIC surfaces. Firstly, we will provide some local results that are valid on BIC surfaces without any additional assumptions on the curvature. Secondly, we will provide global results under a mild global control on the negative part of the curvature, allowing for a large class of non-$\RCD(\kappa,2)/\BICB(\kappa)$ spaces.\vspace{1mm}

To this end, given a complete connected BIC surface $(S,d)$ without cusps and without boundary, together with its two dimensional Hausdorff-measure $\mu$, we prove that: 
\begin{itemize}
\item $(S,d,\mu)$ is an infinitesimally Hilbertian metric measure space in the sense of Ambrosio, Gigli and Savaré \cites{ambrosioCalculus2014, gigliDifferentialStructureMetric2015}, i.e., the Cheeger energy $\Ch$ is a quadratic form, in particular one obtains a canonical self-adjoint Cheeger Laplacian $-\Delta\geq 0$ in $L^2(S,\mu)$,
\item the Carath\'eodory distance $d_{\Ch}$ with respect to the Cheeger energy is equal to the original distance $d$,
\item $(S,d,\mu)$ is locally doubling and its Cheeger energy satisfies a local Poincaré inequality.
\end{itemize}
Combined with Sturm's theory of analysis on strongly local regular Dirichlet spaces \cites{sturmAnalysisLocal1995, sturmAnalysisLocalDirichlet1996}, these results entail the existence of a jointly H\"older continuous and strictly positive heat kernel for $\Delta$, a result which is obviously of fundamental importance for future research on BIC surfaces, e.g. in stochastic analysis. \\
Although it is true \cite{chenUniformConvergenceMetrics2025} that the BIC distance $d$ on $S$ can be approximated by smooth Riemannian distances, this result turns out to be not enough to transfer local Riemannian results to $(S,d)$, as the convergence of the distance functions is only locally uniform. Instead, it turns out that one has to approximate $d$ in a stronger sense by more singular distances: in fact, Burago showed in \cite{buragoBiLipschitzEquivalentAleksandrovSurfaces2005} that $d$ can be approximated in the Lipschitz sense by a sequence of polyhedral distances. In order to make use of this result, we carefully establish that every polyhedral distance actually stems from the length distance function associated with a polyhedrally smooth Riemannian metric, a theory which is based on the abstract framework of Riemannian polyhedra by Eells-Fuglede \cite{eellsHarmonicMapsRiemannian2001}. In order to avoid any possibly cumbersome boundary analysis, we develop a new gluing/localization technique that allows us to implant results from the theory of Riemannian polyhedra conveniently.\vspace{1mm}

Finally, using the Reshetnyak-Huber uniformization theorem, we prove that if the negative part of $\omega$ is in the so-called \emph{Dynkin class} of $(S,d,\mu)$, then $(S,d)$ is globally bi-Lipschitz equivalent to a $\BICB(\kappa)$ space. This result allows us to establish that on such spaces, one has locally uniform volume doubling, a locally uniform Poincaré inequality, and global Gaussian heat kernel bounds, as these results are stable under such an equivalence and valid on $\RCD(\kappa,2)/\BICB(\kappa)$ spaces. This result is much in the spirit of recent results on smooth Riemannian manifolds whose negative part of the Ricci curvature satisfies a Kato/Dynkin assumption (cf. the work of Carron-Mondello-Tewodrose
\cites{carronKatoMeetsBakryEmery2023, carronLimitsKato2024, carronLimitsKato2025, carronStrongBranching2025}, Erbar-Rigoni-Sturm-Tamanini \cite{erbarTamedSpacesDirichlet2020}, Güneysu-Kuwae \cite{gunkuwae}, Rose-Stollmann \cite{roseKatoClass2017}, going back at least to Güneysu-Pallara \cite{gueneysuFunctions2015}).

\vspace{1mm}

The paper is organized as follows: Section \ref{section:background} recalls the necessary background on metric measure spaces, Cheeger energies, BIC surfaces and their curvature measures, the Reshetnyak-Huber uniformization theorem, and polyhedral surfaces. Section \ref{section:heat} contains the local analytic results, Section \ref{section:dynkin} introduces BIC-Dynkin surfaces and proves the bi-Lipschitz deformation result, and discusses an illustrative example of a BIC-Dynkin surface: \emph{a volcano of angle $4\pi$}. Finally, the appendix, Section \ref{section:RiemPolyhedra}, collects the relevant facts on Riemannian polyhedra in the sense of Eells-Fuglede, which underlie the polyhedral-to-BIC transfer used throughout the paper.

\subsection*{Acknowledgments}
The third-named author (M.M.) is particularly grateful to colleagues and friends in Chemnitz for their warm hospitality, and to his PhD advisor, the second-named author (B.G.), for his guidance and support. This research was financially supported by the Deutsche Forschungsgemeinschaft (DFG).

\section{Background}
\label{section:background}

\subsection{Some Notions on Metric Spaces}
\label{section:metricMeasure}
Given a metric space $(X,d)$, we write $B(p,r)$ (resp. $B[p,r]$) for the corresponding open (resp. closed) ball around $p$ of radius $r$ and $\mathcal{H}^\alpha$ denotes the Hausdorff measure of dimension $\alpha$. The symbol $C(X)$ stands for continuous functions and $\LIP(X,d)$ for Lipschitz functions, where for function spaces an index 'c' will always stand for 'compactly' supported and an index 'bs' will always stand for 'boundedly' supported. \vspace{1mm}

The length of a curve\footnote{We understand all curves to be continuous.} $\gamma:[a,b]\to X$ is defined as 
$$
L(\gamma):=\sup\left\{\sum^n_{j=1}d(\gamma(t_{j-1}),\gamma(t_j))\>\>|\>\>\text{$a=t_0<\cdots <t_n=b$ is a partition of $[a,b]$}\right\}
$$
and a curve is called \emph{rectifiable}, if it has a finite length. One obtains the underlying length pseudodistance via
$$
d^L(x,y):=\inf\{L(\gamma)\>|\>\text{$\gamma:[0,1]\to X$ is a rectifiable curve with $\gamma(0)=x$ and $\gamma(1)=y$. }\}
$$
The metric space $(X,d)$ is called a \textit{length space} if $d=d^L$, and \textit{geodesic} if there always exists a shortest curve between any two points. \vspace{1mm}

Given any $f:X\to \mathbb{R}$ the number
\[
\LIP(f) := \sup_{x,y\in X} \frac{|f(y)-f(x)|}{d(x,y)}
\]
denotes its \textit{global Lipschitz constant}. We can also define the \textit{local Lipschitz constant}, or \textit{slope}, by
\[
\lip(f)(x) := \limsup_{y\to x} \frac{|f(y)-f(x)|}{d(x,y)}
\]
for $x\in X$ an accumulation point and $\lip(f)(x):=0$ otherwise. Clearly $\lip(f)(x) \leq \LIP(f)$ for all $x\in X$. \vspace{1mm}

A continuous curve $\gamma:[0,1]\to X$ is said to be \textit{absolutely continuous}, if there exists $f\in L^1[0,1]$ such that
$$
d(\gamma(s), \gamma(t))\leq\int_s^t f(\tau)\dd\tau\quad\text{for all $0\leq s < t \leq 1$}.
$$

The \textit{cone over $(X,d)$} is the metric space given by 
$$
\Cone(X) := ([0,\infty)\times X) /\sim,
$$
where $(0,x)\sim (0,y)$ for any $x,y\in X$, with the distance
\[
d_{\Cone} \big((r_1, x_1), (r_2, x_2)\big) := \sqrt{r_1^2 + r_2^2 - 2r_1r_2\cos\big(\min\big(\pi, d(x_1,x_2)\big)\big)}.
\]
The equivalence class represented by $(0,x)$ is called the \textit{tip} of $\Cone(X)$.\vspace{1mm}

If $(d_n)$ is a sequence of distances on $X$ we say that \emph{$d_n\to d$ in the Lipschitz sense as $n\to\infty$} if for every $\varepsilon>0$ there exists $N$ such that for all $n\geq N$ one has
\[
(1-\varepsilon) d \leq d_n \leq (1+\varepsilon) d.
\]

\subsection{Sobolev calculus and Cheeger form}
\label{section:SobolevSpaces}

A \textit{metric measure space} is a triple $(X,d,\mu)$ with
$(X,d)$ a complete separable metric space and $\mu$ a nonnegative Borel measure on $X$ which is finite on open balls. \vspace{1mm}

The Cheeger energy is defined by
\[
\Ch(f): = \inf\liminf_{n\to\infty}  \int_X \lip(f_n)^2\dd\mu,
\]
where the infimum is taken over every sequence $(f_n)\in\LIP(X,d)\cap L^2(X,\mu)$ such that $f_n\to f$ in $L^2(X,\mu)$. The Cheeger energy yields a notion of \textit{Sobolev space} defined by
\[
W^{1,2}(X,d,\mu)  := \{ f\in L^2(X,\mu) \mid \Ch(f)<\infty \}.
\]
We endow this space with a norm
\[
\lVert f \rVert_{W^{1,2}}^2 :=\lVert f \rVert_{L^{2}}^2  + \Ch(f),
\]
turning $W^{1,2}(X,d,\mu)$ into a Banach space. In addition, $W^{1,2}(X,d,\mu)$ is a dense subspace of $L^2(X,\mu)$. \vspace{1mm}

For any $f\in W^{1,2}(X,d,\mu)$, an element $0\leq G\in L^2(X,d,\mu)$ is said to be an \textit{asymptotic relaxed slope for $f$}, if there exists a sequence $(f_n)\subset \mathrm{Lip}(X,d)$ such that $f_n\to f$ strongly in $L^2(X,\mu)$ and $\mathrm{lip}(f_n)\to\widetilde G$
weakly for some $\widetilde G \in L^2(X,\mu)$ with $\widetilde G \leq G$. Then one has
$$
W^{1,2}(X,d,\mu)=\{f\in L^2(X,\mu)\>\>|\text{ $f$ has an asymptotic relaxed slope}\}.
$$
Given $f\in W^{1,2}(X,d,\mu)$ there exists a minimal element for the $L^2$-norm of the set of asymptotic relaxed slopes which is denoted by $|Df|$ and called the \textit{minimal weak upper gradient of $f$}. This element is also minimal $\mu$-a.e., and there holds
\[
\Ch(f) = \int_X |Df|^2\dd\mu\quad\text{for all $f\in W^{1,2}(X,d,\mu)$}.
\]
One has the locality rule
$$
|Df_1|=|Df_2|\quad \text{$\mu$-a.e. in $\{f_1=f_2\}$, for all $f_1,f_2\in W^{1,2}(X,d,\mu)$.}
$$
It is straightforward to see that if $d'$ is another distance function on $X$ and $\mu'$ is another measure on $X$ such that
$$
(1/\alpha) d'\leq d \leq \alpha d'\quad\text{for some constant $\alpha\geq 1$},
$$
and
$$
(1/\beta) \mu'\leq \mu \leq \beta \mu'\quad\text{for some constant $\beta\geq 1$},
$$
then $W^{1,2}(X,d,\mu)=W^{1,2}(X,d',\mu')$ with equivalent norms and moreover, with an obvious notation,
\begin{align}\label{ecpoop}
(1/\alpha)|Df|_{d',\mu'}\leq |Df|_{d,\mu}\leq \alpha|Df|_{d',\mu'}\quad\text{$\mu/\mu'$ a.e. in $X$.}
\end{align}

The space $ (X,d,\mu)$ is said to be \textit{infinitesimally Hilbertian} if $W^{1,2}(X,d,\mu)$ is a Hilbert space, i.e.\, $\Ch$ is a quadratic form. This is equivalent to the validity of the parallelogram identity
$$
|D(f_1+f_2)|^2+|D(f_1-f_2)|^2=2|Df_1|^2+2|Df_2|^2\quad\text{for all $f_1,f_2\in W^{1,2}(X,d,\mu)$}.
$$
In the infinitesimally Hilbertian case, $\Ch$ with the definition domain $W^{1,2}(X,d,\mu)$ is a densely defined closed nonnegative quadratic form in $L^2(X,\mu)$, so by Kato's theory there is a uniquely determined nonpositive self-adjoint operator $\Delta$ in $L^2(X,\mu)$, the \emph{Cheeger Laplacian}, such that $\dom(\Delta)\subset W^{1,2}(X,d,\mu)$ and 
$$
\langle -\Delta f_1,f_2\rangle_{L^2}=\Ch(f_1,f_2)\quad\text{for all $f_1\in\dom(\Delta)$, $f_2\in W^{1,2}(X,d,\mu)$}.
$$

The local Sobolev space $W^{1,2}_\loc(X,d,\mu)$ is defined to be the space of all $f\in L^2_\loc(X,\mu)$ such that for all bounded $B\subset X$ there exists $f_B\in W^{1,2}(X,d,\mu)$ such that $f=f_B$ $\mu$-a.e. on $B$. For any $f\in W^{1,2}_\loc(X,d,\mu)$ one sets $|Df|:=|Df_B|$ on $B$ for any bounded $B$ and any $f_B\in W^{1,2}(X,d,\mu)$ such that $f=f_B$ $\mu$-a.e. on $B$. This is well defined by the locality rule.

\begin{theorem}\label{regula} Assume that $(X,d,\mu)$ is infinitesimally Hilbertian. Then $\Ch$ with the domain of definition $W^{1,2}(X,d,\mu)$ is a strongly local Dirichlet form in $L^2(X,\mu)$ that has $\mathrm{Lip}_{\mathrm{bs}}(X,d)$ as the core of the form. In other words, $\mathrm{Lip}_{\mathrm{bs}}(X,d)\subset W^{1,2}(X,d,\mu)$ is dense. In particular, if $(X,d)$ is locally compact and proper, then $\Ch$ with the domain of definition $W^{1,2}(X,d,\mu)$ is a strongly local regular Dirichlet form in $L^2(X,\mu)$.
\end{theorem}

\begin{proof} As $W^{1,2}(X,d,\mu)$ is a Hilbert space, it is uniformly convex and thus
the density in energy of $\mathrm{Lip}_{\mathrm{bs}}(X,d)$, see \cite{ags} and also the recent \cite{eriksson-biqueDensityLipschitzFunctions2023}, transforms into norm density in $W^{1,2}(X,d,\mu)$.
\end{proof}

As with any strongly local Dirichlet form, associated to $\Ch$, we can define a pseudo-distance $d_{\Ch}$, called \textit{Carathéodory distance}, by
\[
d_{\Ch}(x,y) := \sup\{ f(y) - f(x) \mid f\in W^{1,2}_\loc(X,d,\mu)\cap C(X)~\text{s.t.}~|Df| \leq 1\}.
\]

\subsection{BIC surfaces and subharmonic metrics}
Let $(X,d)$ be a metric space and let $x\in X$.
Let $\gamma, \tilde\gamma: [0,1]\to X$ be two curves with $\gamma(0)=\tilde\gamma(0)=x$.
For any $t\in[0,1]$, there is a unique triangle (up to isometry) in the Euclidean plane such that the three edges have length
$d(x,\gamma(t)), d(\gamma(t), \tilde\gamma(t))$ and $d(\tilde\gamma(t), x)$.
We denote by $\bar\alpha(t)$ the angle at the vertex opposite to the side $d(\gamma(t), \tilde\gamma(t))$ in this Euclidean model triangle.
With that we can define the \textit{upper angle} $\alpha$ between $\gamma$ and $\tilde\gamma$ at $x$ by
\[
\alpha := \limsup_{t\to 0} \bar\alpha(t).
\]
A \textit{triangle} in $(X,d)$ is the data of three points and three minimizing geodesics joining them.
For a triangle $T$, the \textit{upper excess} is given by
\[
\delta(T) := \alpha + \beta + \gamma - \pi
\]
where $\alpha,\beta,\gamma$ are the upper angles at the vertices of $T$. When $\mathcal{T}$ is a finite family of non-overlapping triangles, then
\[
\delta(\mathcal{T}) := \sum_{T\in\mathcal{T}} \delta(T),
\]
where $\delta(\mathcal{T})=0$ for $\mathcal{T}=\varnothing$. A triangle $T$ is said to be \textit{convex relative to the boundary} if no couple of points on the boundary of $T$ can be joined by a curve lying outside $T$ and shorter than the part of the boundary joining the points. A triangle is said to be \textit{simple} if it is homeomorphic to a disc and convex relative to the boundary.\vspace{1mm}

In the sequel, we understand all manifolds to be without boundary - unless stated otherwise - and a smooth surface (with boundary) is understood to be a smooth $2$-dimensional manifold (with boundary).

\begin{definition}
A length space $(S,d)$ is called a \textit{surface with locally bounded integral curvature}, in short, a \textit{BIC surface}, if
\begin{itemize}
\item $S$ is a smooth connected surface,
\item $d$ induces the original topology of $S$,
\item for every compact $K$ there exists a constant $C=C(K)>0$ such that $\delta(\mathcal{T})\leq C$
for any finite collection $\mathcal{T}$ of simple non-overlapping triangles in $(S,d)$ contained in $K$.
\end{itemize}
\end{definition}
The central object associated with such a singular surface is its \textit{curvature measure}, which is a signed measure $\omega$ on $S$ defined as follows: with the usual conventions for the positive and the negative part, $(\cdot)^+=\max(0,\cdot)$ and $(\cdot)^-=-\min(0,\cdot)$, on any open $U\subset S$ one defines
\[
\omega^\pm(U) := \sup_{\mathcal{T}} \sum_{T\in\mathcal{T}}\delta(T)^\pm,\quad
\]
with the suprema taken on finite families $\mathcal{T}$ of non-overlapping simple triangles contained in $U$. We extend $\omega^\pm$ to an arbitrary Borel subset $E\subset S$ by
\[
\omega^\pm(E) := \inf_{U\supset E~\text{open}} \omega^\pm(U)
\]
and set 
$$
\omega := \omega^+- \omega^-.
$$
It is straightforward that $\omega^+$ is locally finite, but for $\omega^-$ this is a highly non-trivial fact of the theory, see \cite[Theorem 15 p.134]{AleksandrovIntrinsicGeometry}.
In view of \cite[Theorem 2.18]{zbMATH03228674} the curvature measure is ultimately a signed Radon measure.\vspace{1mm}

A point $p\in S$ with $\omega(\{p\})= 2\pi$ is called a \emph{cusp}.
This name comes from the fact that if $\omega(\{p\})= 2\pi$, then $p$ can be interpreted as a point with no angle around it, see \cite[Theorem 8 p.166]{AleksandrovIntrinsicGeometry}.\vspace{2mm}

\begin{remark}
It is a nontrivial fact that the $2$-dimensional Hausdorff measure on $(S,d)$ is locally finite.
Indeed, by \cite{zbMATH03218459} the 2-dimensional Hausdorff measure coincides with the area measure defined in \cite[Chapter VIII]{AleksandrovIntrinsicGeometry}.
The latter is locally finite by \cite[Theorem 2 p.261]{AleksandrovIntrinsicGeometry}.
\end{remark}

We now come to another, more analytic, description of BIC surfaces: assume $(S,h)$ is a smooth Riemannian surface with $\Delta_h$ the associated Laplace-Beltrami operator, $\mu_h$ the volume measure and $d_h$ the Riemannian distance. Given $u\in L^1_\loc(S,\mu_h)$, we say that a signed Radon measure $\nu$ is the \textit{distributional Laplacian of $u$ with respect to $h$}, if 
$$
\int_S (\Delta_h\varphi)  u\dd\mu_h = \int_S \varphi\dd\nu\quad\text{for all $\varphi\in C_c^\infty(S)$},
$$
where $\mu_h$ denotes the Riemannian volume measure. In that case, $\nu$ is uniquely determined and one sets $\underline{\Delta}_h u:=\nu$. Integrating by parts, one gets the consistency
$$
\underline{\Delta}_h u= (\Delta_h u)\mu_h\quad\text{for all $u\in C^2(S)$}.
$$
We define $\mathcal{V}(S,h)$ as the set of functions in $L^1_\loc(S,\mu_h)$ whose distributional Laplacian is a signed Radon measure. We record that (see \cite[Remark 4.56]{fillastreSubharmonic2023}) for all $q\in[1,2)$ one has 
$$
\mathcal{V}(S,h)\subset W^{1,q}_{\mathrm{loc}}(S,d_h,\mu_h)
$$

Note that these local function spaces $L^q_\loc (S,\mu_h)$ and $W^{1,q}_{\mathrm{loc}}(S,d_h,\mu_h)$ of course do not depend on $h$ and are actually defined via their local Euclidean counterparts through smooth charts.\vspace{1mm}

The pseudo-distance $d_{h,u}$, with $u\in\mathcal{V}(S,h)$, defined for any $x,y\in S$, by
\[
d_{h,u}(x,y) := \inf\left\{ \int_0^1 \ee^{u(\gamma(t))}|\gamma'(t)|_h\dd t\>\> \middle|\>\>
\gamma:[0,1]\to (S,d_h)~\text{Lipschitz s.t.}~\gamma(0)=x, \gamma(1)=y \right\},
\]
is called a \textit{subharmonic distance}. Note that this definition uses that any $u\in\mathcal{V}(S,h)$ is well-defined away from a polar set $N_u$ (Hausdorff dimension 0) such that $\gamma^{-1}(N_u)$ is a null set in $[0,1]$ for every Lipschitz curve $\gamma:[0,1]\to (S,d_h)$, see \cite[Lemma 4.72]{fillastreSubharmonic2023}.
\vspace{2mm}

The following theorem highlights the equivalence between BIC surfaces and subharmonic distances:

\begin{theorem}[Reshetnyak-Huber]\th\label{thm:ReschetnyakHuber}
Let $S$ be a connected smooth surface.\\
\emph{(i)} If $h$ is a smooth Riemannian metric on $S$ and $u\in\mathcal{V}(S,h)$ are such that $d_{h,u}$ has no point at infinite distance
(that is, a point for which every path containing it has infinite length),
then $(S,d_{h,u})$ is a BIC surface, the underlying $2$-dimensional Hausdorff measure $\mu_{h,u}$ is given by 
$$
\dd \mu_{h,u}= \ee^{2u} \dd\mu_h,
$$
and its curvature is given by 
\[
\omega_{h,u} = \mathbb{K}_h \mu_h + \underline{\Delta}_h u,
\]
where $\mathbb{K}_h\in C^\infty(S)$ is the Gaussian curvature of $h$. \\
\emph{(ii)} Conversely, if $(S,d)$ is a BIC surface, then there exist a smooth Riemannian metric $h$ on $S$ and a function $u\in\mathcal{V}(S,h)$ such that $d=d_{h,u}$. 
\end{theorem}

\begin{proof}
See Theorem 2.23 in \cite[Chapter 2]{fillastreSubharmonic2023} or \cite[Theorem 7.3]{troyanov2022alexandrovssurfacesboundedintegral}.
The equality between $2$-dimensional Hausdorff measure and the integral of the volume form is given by \cite[Theorem 4.176]{fillastreSubharmonic2023}.
\end{proof}

These facts show that it is reasonable to consider $\ee^{2u}h$ as a (nonsmooth) Riemannian metric on $S$.

\begin{definition}
Let $(S,d)$ be a BIC surface with its curvature measure $\omega$ and its $2$-dimensional Hausdorff measure $\mu$. Let $\kappa\in\mathbb{R}$. We say that $(S,d)$ has \textit{curvature bounded below by $\kappa$}, and shortly $(S,d)$ is $\BICB(\kappa)$ if one has $\omega\geq\kappa\mu$ in the sense of set functions.
\end{definition}

\begin{remark}\label{edpopo} It is a highly nontrivial recent result by the third-named author (\cite[Theorem 1]{marot2025notionscurvaturesingularsurfaces}) that all complete members of $\BICB(\kappa)$ without cusps are members of $\mathrm{CBB}(\kappa,2)$ - the class of complete metric spaces having Hausdorff dimension $2$ and curvature bounded from below by $\kappa$ in the sense of Alexandrov. On the other hand, it has been shown in \cite{petruninAlexandrovMeetsLottVillaniSturm2011} that every surface in $\mathrm{CBB}(\kappa,2)$ together with its $2$-dimensional Hausdorff measure lies in $\mathrm{RCD}(\kappa,2)$ - the class of infinitesimally Hilbertian metric measure spaces having synthetic Ricci curvature bounded from below by $\kappa$ and synthetic dimension $\leq 2$.
\end{remark}

\subsection{Polyhedral surfaces}

In this section we treat a particular case of BIC surfaces called polyhedral surfaces. These surfaces are roughly speaking made of flat triangles glued isometrically along the edges and they will serve as a good approximating model of BIC surfaces.

\begin{definition}
a) Let $\theta>0$ and let $h>0$. We denote by $A(\theta,h)$ the subset of $\Cone([0,\theta])$ consisting of points $(r,x)$ such that $r<h$.
The metric space $A(\theta,h)$ is called \textit{truncated cone over a segment of angle $\theta$ and radius $h$}.\\
b) Let $C(\theta)\subset \R^2$ denote the circle of length $\theta>0$ around $0$.
We denote by $Q(\theta, h)$ the subset of $\Cone(C(\theta))$ consisting of points $(r,x)$ such that $r<h$.
The set $Q(\theta,h)$ is called \textit{truncated cone over a circle of angle $\theta$ and height $h$}.
\end{definition}

To visualize geometrically, when $\theta\leq 2\pi$ the space $A(\theta, h)$ is a cone in the plane of angle $\theta$
and $Q(\theta, h)$ is a cone in $\R^3$ of angle $\theta$. When $\theta>2\pi$, neither cone can be isometrically embedded in Euclidean space.

\begin{definition}\label{polyc} Let $P$ be a connected smooth surface with boundary, endowed with a distance $d$. We say that $(P,d)$ is a \textit{polyhedral surface} and that $d$ is a \textit{polyhedral distance}, if the following conditions hold:
\begin{enumerate}[label=(\roman*)]
\item $d$ is a length metric that induces the original topology of $P$,
\item every interior point $p\in P$ has a neighborhood isometric onto a cone of type $Q(\theta, h)$ where $\theta>0$ and $h>0$
and $p$ is mapped to the tip,
\item every boundary point $p\in P$ has a neighborhood isometric onto a cone of type $A(\theta, h)$ where $\theta>0$ and $h>0$
and $p$ is mapped to the tip.
\end{enumerate}
\end{definition}

The neighborhoods in (ii) and (iii) are called \textit{conical model neighborhoods}, and the angles $\theta$ in (ii) and (iii) do not depend on the neighborhood around a point $p\in P$. Indeed, let $p\in P$ and $U$ be a neighborhood isometric to $A(\theta, h)$ for some $\theta>0$ and $h>0$.
Then for $r<h$, the length of the circle centered at $p$ and with radius $r$, written $l(p,r)$ is given by $\theta r$.
Thus we obtain that
\[
\theta = \lim_{r\to 0} \frac{l(p,r)}{r}.
\]
This shows that $\theta$ is defined independently of $U$. The same reasoning applies for $Q(\theta, h)$ and by taking the length of an arc. The quantity $\theta(p)$, defined above, is called the \textit{total angle at $p$}.\\
If $p\in P$ is an interior point and $\theta(p)\neq 2\pi$ then $p$ is called a \textit{surface vertex} and if $p$ is a boundary point and $\theta(p)\neq\pi$ then $p$ is called a \textit{boundary surface vertex}. Note that surface vertices and boundary surface vertices are isolated points.\vspace{1mm}

\begin{remark}\label{espo} Given a smooth connected surface with boundary $P$ there exists a unique LBL structure on $P$ (cf. Definition \ref{lbldef}) given by any smooth Riemannian distance on $P$, and it is straightforward to check this is the same LBL structure given any polyhedral distance on $P$. We will refer to this LBL structure as the canonic LBL structure on $P$. 
\end{remark}

For any boundary point $p\in \partial P$, if any, we define the \textit{turn of the boundary at $p$} to be the quantity 
$$
\kappa(p):=\pi-\theta(p).
$$
First, for any interior point $p\in P$, we define the \textit{curvature at} $p$ as the quantity
\[
\omega(p) := 2\pi - \theta(p).
\]
With $V$ the set of surface vertices, for any Borel set $E\subset P$, we define
\[
\omega(E) := \sum_{p\in V} \omega(p)\delta_p(E),
\]
where $\delta_p$ is the Dirac measure at $p$. Likewise, one defines
$$
\kappa(E):=\sum_{p\in bV} \kappa(p)\delta_p(E),
$$
where $bV$ denotes the set of boundary surface vertex points. Both $\omega$ resp. $\kappa$ are signed Radon measures on $P$, called \emph{polyhedral curvature measure} resp. \emph{polyhedral turn measure}.

\begin{theorem}[Gauss-Bonnet for polyhedral surfaces] For every compact polyhedral surface $(P,d)$ one has
\[
\omega(P) + \kappa(\partial P) = 2\pi\chi(P).
\]
\end{theorem}

\begin{proof}
See \cite[Theorems 5.3.1 and 5.3.2]{ReshetnyakTwoDimManifolds}.
\end{proof}

\begin{proposition} Every polyhedral surface $(P,d)$ without boundary is a BIC surface without cusps.
\end{proposition}

\begin{proof}
Let $K\subset P$ be a neighborhood homeomorphic to a disc.
Then $K$ can be covered by a finite number of conical model neighborhoods.
Thus $K$ contains a finite number of vertices.
Let $(T_i)_{1\leq i \leq n}$ be a finite family of simple non-overlapping geodesic triangles in $K$.
By Gauss-Bonnet for polyhedral surfaces,
\[
\omega(T_i) = \delta(T_i).
\]
As $(T_i)$ is non-overlapping, every point is contained in at most one triangle.
Thus we have
\[
\sum_{i=1}^n \delta(T_i) \leq \sum_{i=1}^n \omega^+(T_i) \leq \omega^+(K)
\]
However, as the curvature is concentrated on surface vertices, this shows that $\omega^+(K)$ is upper bounded by a constant depending only on $K$.
\end{proof}

It is clear that the polyhedral curvature measure $\omega$ defined in this paragraph coincides with the curvature measure defined for a general BIC surface: Indeed, it suffices to look at the curvatures in conical neighborhoods to see that both are concentrated on the tip of the cone.\vspace{2mm}

We now make contact with abstract locally Lipschitz polyhedra as defined in the appendix of this paper:

\begin{theorem}
\th\label{bicSurf:th:polyhedronTriang}
Every polyhedral surface $(P,d)$ without boundary admits a locally Lipschitz triangulation $T=(K,\vartheta)$ such that $K$ is homogeneous of dimension $2$ and for all $s\in \mathcal{S}^{(2)}(P,T)$ there exists an affine coordinate system $\phi_s$ for $s$ such that the composition
$$
\vartheta \circ \phi^{-1}_s:\R^2\supset\phi_s(\overline{s})\longrightarrow \vartheta(\overline{s})\subset P
$$
is a metric isometry. 
\end{theorem}

\begin{proof}
See \cite[Theorem 5.2.2]{ReshetnyakTwoDimManifolds}.
\end{proof}

In particular, every polyhedral surface without boundary (with its canonic LBL structure) is a locally Lipschitz polyhedron that is homogeneous of dimension $2$.

\begin{theorem}\label{letsa} Every polyhedral surface $(P,d)$ without boundary admits a polyhedrally smooth Riemannian metric $g$ (cf. \th\ref{riemPolyhedra:th:EqTwoDefSobolev}) such that $d=d_g$.
\end{theorem}

\begin{proof}Pick $T=(K,\vartheta)$ as in Theorem \ref{bicSurf:th:polyhedronTriang} and $h$ and $u$ such that $d=d_{h,u}$. Let $V$ denote the countable set of surface vertices of $(P,d)$. Given any $s\in \mathcal{S}^{(2)}(P,T)$ we have $\vartheta(s^\circ)\subset P\setminus V$. On the open set $P\setminus V$ we have
$$
 \underline{\Delta}_h u=-\mathbb{K}_h \mu_h, 
$$
and so $u$ is smooth on this set by local elliptic regularity, and for any affine coordinate system $\phi_s$ as in Theorem \ref{bicSurf:th:polyhedronTriang}, the map
$$
 \phi_s\circ \vartheta^{-1}: \vartheta(s^\circ)\longrightarrow \phi_s(s^\circ)
$$
becomes a smooth chart such that 
\begin{align}\label{dio}
(\phi_s\circ \vartheta^{-1})^*h_{\mathrm{Eucl}}=\ee^{2u}h,
\end{align}
It follows that $g^T_{s,ij}(x):=\delta_{ij}$ does the job: indeed, if $\gamma:[0,1]\to P$ is a Lipschitz curve (cf. Remark \ref{espo}), then for its polyhedral $g$-length (cf. Appendix) one has
$$
\mathscr{L}_g(\gamma) = \sum_{s\in S^{(m)}(P,T)} \bigintsss_{\gamma^{-1}(\vartheta(s^\circ))}
\sqrt{\sum_{1\leq j\leq m}   \dot\gamma_{\phi_s}^j\dot\gamma_{\phi_s}^j }\dd t,
$$
where
$$
\gamma_{\phi_s}:=\phi_s\circ \vartheta^{-1} \circ \gamma.
$$
On the other hand, as $\vartheta$ is locally Lipschitz, with
$$
N:=\bigsqcup_{s\in S^{(m)}(P,T)}\vartheta(s^\circ)
$$
one has $\mu_h\left(P\setminus N\right)=0$, and so as $\gamma$ is Lipschitz, one has that $\gamma^{-1}(P\setminus N)\subset [0,1]$ has Lebesgue measure $0$. It follows that 
$$
\int_0^1 \ee^{u(\gamma(t))}|\gamma'(t)|_h\dd t=\sum_{s\in S^{(m)}(P,T)} \bigintsss_{\gamma^{-1}(\vartheta(s^\circ))} \ee^{u(\gamma(t))}|\gamma'(t)|_h\dd t,
$$
which in view of (\ref{dio}) is $=\mathscr{L}_g(\gamma)$, and finishes the proof.
\end{proof}

\begin{remark} With similar arguments one finds that in the above situation $\mu_{h,u}$ is equal to the polyhedral Riemannian volume measure $\mu_g$ (cf. appendix). In particular, it follows that complete polyhedral surfaces without boundary together with their $2$-dimensional Hausdorff measures are infinitesimally Hilbertian, locally doubling and satisfy local Poincaré estimates (cf. Remark \ref{rypo}).
\end{remark}

\section{Functional inequalities and heat kernels on complete BIC surfaces without cusps}\label{section:heat}

Throughout this section, let $(S,d)$ be a complete BIC surface without cusps, with $\mu$ denoting its $2$-dimensional Hausdorff measure. \vspace{3mm}

The following result will ultimately allow us to transfer analytic results from Riemannian polyhedra to BIC surfaces:

\begin{lemma}
\th\label{lem:neighborEmbed}
For any point $x\in S$ and any compact geodesically convex neighborhood $P\subset S$ of $x$ homeomorphic to a disc, there exists a closed BIC surface $(Y,e)$ without cusps such that
\begin{itemize}
\item $Y$ is diffeomorphic to $\mathbb{S}^2$,
\item $(P, d)$ is isometrically embedded in $(Y,e)$,
\item there exists a sequence $(e_i)$ of polyhedral distances on $Y$ converging to $e$ in the Lipschitz sense,
\item $\omega_i^\pm\rightharpoonup\omega^\pm$, where $\omega_i$ (resp. $\omega$) is the curvature measure of $e_i$ (resp. $e$).
\end{itemize}
\end{lemma}

Note that such a neighborhood $P$ of $x$ always exists by \cite[Theorem 1 p.58]{AleksandrovIntrinsicGeometry}, \cite[Theorem 4.122]{fillastreSubharmonic2023} and \cite[Theorem 8.1.8]{ReshetnyakTwoDimManifolds}.

\begin{proof}
As $P$ is homeomorphic to a disc, in the sequel, we can consider that $P$ is the closed unit disc of $\mathbb{R}^2$
endowed with the BIC distance $d$ by isometry.
We define $Y$ as the gluing along the boundaries of a copy of $P$ and a hemisphere $H\subset\mathbb{S}^2$ with the same radius as the boundary of $P$, endowed with the length distance $e$ induced by the respective distances on $H$ and $P$.
By \cite[Theorem 8.3.1]{ReshetnyakTwoDimManifolds}, $Y$ is a closed BIC surface and, moreover, since $P$ is the disc,
$Y$ is diffeomorphic to the sphere $\mathbb{S}^2$.
The surface $Y$ does not contain any cusps as the total angle is everywhere positive.
We can also justify it using excision and pasting from \cite[Chapter VI]{AleksandrovIntrinsicGeometry}.
The existence of the approximating sequence by polyhedral distances is guaranteed by Burago's theorem \cite[Lemma 6]{buragoBiLipschitzEquivalentAleksandrovSurfaces2005}.
\end{proof}

\begin{theorem}\th\label{thm:AnalysisOnBIC} \emph{(i)} The metric measure space $(S,d,\mu)$ is infinitesimally Hilbertian.\\
\emph{(ii)} One has $d_{\Ch} = d$.\\
\emph{(iii)} For any point $x\in S$ there exists an open neighborhood $U$ of $x$ such that one has a weak Poincaré inequality on $U$, i.e.\ there exists $C,k>0$, depending on $U$, such that for any $f\in W^{1,2}(S,d,\mu)$ and any $z\in U$ with $B(z,kr)\subset U$ one has 
\[
\int_{B(z, r)}|f-f_B|^2\dd\mu \leq Cr^2 \int_{B(z,kr)} |Df|^2\dd\mu,\quad \text{where $u_B := \fint_{B(z,r)}f\dd\mu$}.
\]
\emph{(iv)} For any point $x\in S$ there exists an open neighborhood $U$ of $x$ such that $\mu$ is doubling on $U$, i.e. there exists $c>0$, depending on $U$, such that for any $z\in U$ and every $r>0$ with $B(z,2r)\subset U$ one has
\[
\mu(B(z, 2r)) \leq c \mu(B(z,r)).
\]
\end{theorem}

\begin{proof}
(i)
Let $f_1,f_2\in W^{1,2}(S,d,\mu)$.
Let $x\in S$ and let $U$ be a neighborhood of $x$ contained in some $P$ chosen as in \thref{lem:neighborEmbed}.
We consider a smooth cut-off function $\varphi : P\to [0,1]$ such that $\varphi=1$ on $U$ and
$$
d(\supp(\varphi), \partial P)>0.
$$
Let $(Y,e)$ be an embedding space of $P$ as in \thref{lem:neighborEmbed}. As 
$$
d(\supp(\varphi f_j), \partial P)>0,
$$
by \cite[Proposition 2.6]{gigliDifferentialStructureMetric2015} one has $\varphi f_j\in W^{1,2}(Y, e,\mu_e)$, where $\mu_e$ denotes the $2$-dimensional Hausdorff measure on $(Y,e)$. By locality of the minimal weak upper gradient
\begin{align}\label{edqqa}
|D(\varphi f_j)| = |D f_j|,\quad |D(\varphi (f_1\pm f_2))| = |D (f_1\pm f_2)|\quad\text{ $\mu$-a.e. on $U$. }
\end{align}
Pick a sequence of polyhedral distances $(e_i)$ on $Y$ as in \thref{lem:neighborEmbed}. By Remark \ref{rypo} from the appendix we know that $(Y, e_i,\mu_{e_i})$ is infinitesimally Hilbertian.
So, for all $0<\varepsilon<1$, using bi-Lipschitz equivalence, (\ref{ecpoop}) and the locality (\ref{edqqa}) of the minimal weak upper gradient, we see that inequalities
\[
\left(\frac{1-\varepsilon}{1+\varepsilon}\right)^2 \big( 2|Df_1|^2 + 2|Df_2|^2 \big)
\leq |D(f_1+f_2)|^2 + |D(f_1-f_2)|^2
\leq \left(\frac{1+\varepsilon}{1-\varepsilon}\right)^2 \big( 2|Df_1|^2 + 2|Df_2|^2 \big)
\]
hold $\mu$-a.e. on $U$, and the proof is completed by sending $\varepsilon\to 0^+$.\vspace{1mm}

(ii) Fix $x,y\in S$. Since $d(x,\cdot)$ is $1$-Lipschitz, we have $d(x,\cdot)\in W^{1,2}_\loc(S,d,\mu)\cap C(S)$ with $|Dd(x,\cdot)|_d\leq 1$ $\mu$-a.e. and hence,
\[
d_{\Ch}(x,y) \geq d(x,y) - d(x,x) = d(x,y).
\]
So $d_{\Ch}\geq d$.
For the converse inequality, let $f\in W^{1,2}_\loc(S,d,\mu)\cap C(S)$ be such that $|Df| \leq 1$ $\mu$-a.e. on $S$.
Take any geodesic $\gamma:[0,1]\to S$ that joins $x$ to $y$.
Let $\{U_i\}_{0\leq i\leq k}$ be a finite cover of $\gamma$ by successively overlapping open sets
such that each $U_i\subset P_i$ where $P_i$ is chosen as in \thref{lem:neighborEmbed}.
Also, let $(t_i)_{0\leq i\leq k}$ be an increasing sequence in $[0,1]$ such that $t_0 = 0, t_k = 1$
and $\gamma(t_i)\in U_{i-1}\cap U_{i}$ for $0<i<k$. We claim that $f\in \LIP(U_i,d)$. First, let us fix an $i\in\{0,1,\dots, k\}$. As in (i), we consider a smooth cut-off function $\varphi_i= 1$ on a neighborhood of $U_i$ with
$d(\supp\varphi_i, \partial P_i)>0$.
Then $\varphi_i f\in W^{1,2}(Y_i, e_i, \mu_{e_i})$, where $(Y_i,e_i)$ is an embedding space of $P_i$ as in \thref{lem:neighborEmbed}. For any $0<\varepsilon<1$, considering $(e_{ij})_{j\geq 1}$ to be a Lipschitz approximation of $e_i$
by polyhedral distances in $Y_i$, using locality of the weak minimal upper gradient we get that
\[
|Df|_{e_{ij}} \leq \frac{1}{1-\varepsilon}|Df|_{e_i} \leq \frac{1}{1-\varepsilon}
\]
holds $\mu$-a.e. on $U_i$ and $j$ large enough.
But, by \cite[Proposition 4.1]{eellsHarmonicMapsRiemannian2001}, $f$ will be $1/(1-\varepsilon)$-Lipschitz for $e_{ij}$ on $U_i$.
So, in turn, $f$ is $(1+\varepsilon)/(1-\varepsilon)$-Lipschitz for $d$ on $U_i$.
As $\varepsilon$ is arbitrary, by letting it tend to $0^+$, $f$ is $1$-Lipschitz for $d$ on $U_i$.
Thus, we have
\begin{align*}
f(x) - f(y) &= \sum_{i=1}^k f(\gamma(t_i)) - f(\gamma(t_{i-1})) \\
&\leq \sum_{i=1}^k d(\gamma(t_i), \gamma(t_{i-1})).
\end{align*}
So, when $k\to\infty$, as $(S,d)$ is a length space, this procedure leads us to
\[
f(x) - f(y) \leq d(x,y),
\]
and therefore $d_{\Ch}\leq d$.\\
(iii) Let $x\in S$. Again pick a neighborhood $P$ of $x$ as in \thref{lem:neighborEmbed} and consider a $\rho>0$ such that $B(x,2\rho)\subset P$ and $2\rho<d(x,\partial P)$. We set $U:=B(x,\rho)$ and let $(Y,e)$ be an embedding space of $P$ with Lipschitz approximation $(e_i)_{i\geq 1}$. Now given $0<\varepsilon<1$ arbitrary we pick an index $i$ large enough such that 
\[
(1-\varepsilon)e \leq e_i \leq (1+\varepsilon)e
\]
on $Y$. Remark \ref{rypo} gives us constants $\tilde{C}>0$ and $\tilde{k}>1$ such that if $v\in W^{1,2}(Y,e_i,\mu_{e_i})$ and if $B_{e_i}(z,kr)\subset B_e(x, 2\rho)$ then one has
\[
\int_{B_{e_i}(z, r)}|v-v_{B_{e_i}}|^2\dd\mu_{e_i} \leq \tilde{C}r^2 \int_{B_{e_i}(z, \tilde{k} r)} |Dv|_{e_i}^2\dd\mu_{e_i}
\]
with $v_{B_{e_i}} := \fint_{B_{e_i}(z,r)}v\dd\mu_{e_i}$. In order to fit everything in $B(x,2\rho)$, we restrict the radius $r$ to the interval
\[
r \in \Big(0, \frac{\rho(1-\varepsilon)}{\tilde{k}(1+\varepsilon)} \Big).
\]
Now we take any $f\in W^{1,2}(S, d, \mu)$, any $z\in U$ and $r$ as above. Again, we consider a smooth cut-off function $\varphi=1$ on a neighborhood of $U$ with $d(\supp(\varphi), \partial P)>0$.
Then, $v=\varphi f\in W^{1,2}(Y, e, \mu_{e})$. Using the locality of the weak minimal upper gradient, we now estimate with an obvious notation as follows:
\begin{align*}
&\phantom{=}~\int_{B(z,r)} |f-f_{B}|^2\dd\mu\\
&=\int_{B_e(z,r)} |f-f_{B_e}|^2\dd\mu_e\\
&\leq 4 \int_{B_e(z,r)} |f-f_{B_{e_i}}|^2\dd\mu_{e} \\
&\leq 4 \frac{1}{(1-\varepsilon)^2} \int_{B_e(z,r)} |f-f_{B_{e_i}}|^2\dd\mu_{e_i}\\ &\leq 4 \frac{1}{(1-\varepsilon)^2} \int_{B_{e_i}(z,(1+\varepsilon) r)} |f-f_{B_{e_i}}|^2\dd\mu_{e_i} \\
&\leq 4 \tilde{C} \frac{(1+\varepsilon)^2}{(1-\varepsilon)^2} r^2 \int_{B_{e_i}(z, \tilde{k}(1+\varepsilon) r)} |Df|_{e_i}^2\dd\mu_{e_i} \\
&\leq 4 \tilde{C} \frac{(1+\varepsilon)^4}{(1-\varepsilon)^2} r^2 \int_{B_e(z, \tilde{k}(1+\varepsilon)(1-\varepsilon)^{-1} r)} |Df|^2_e\dd\mu_e\\
&= 4 \tilde{C} \frac{(1+\varepsilon)^4}{(1-\varepsilon)^2} r^2 \int_{B(z, \tilde{k}(1+\varepsilon)(1-\varepsilon)^{-1} r)} |Df|^2\dd\mu
\end{align*}
completing the proof, if we set 
$$
U:=B\Big(x, \frac{\rho(1-\varepsilon)}{\tilde{k}(1+\varepsilon)}\Big),\quad k:=\tilde{k}(1+\varepsilon)(1-\varepsilon)^{-1}. 
$$

(iv) Similarly, we place ourselves in the same setting as (iii) and since
we want to stay in the ball $B(x,2\rho)$, we restrict $r$ to 
\[
r\in \Big(0, \frac{\rho}{2(1+\varepsilon)}\Big).
\]
Then, using once more Remark \ref{rypo} for any $z\in B(x,\rho)$ and any $r$ as above one has
\begin{align*}
&\mu(B(z,2r))=\mu_e(B_e(z,2r))\leq \frac{1}{(1-\varepsilon)^2} \mu_{e_i}(B_e(z,2r))\\
&\leq \frac{1}{(1-\varepsilon)^2} \mu_{e_i}(B_{e_i}(z,2(1+\varepsilon)r))\leq 4\widetilde c \frac{(1+\varepsilon)^2}{(1-\varepsilon)^4} \mu_{e_i}(B_{e_i}(z,(1-\varepsilon)r))\\
&\leq 4\widetilde c \frac{(1+\varepsilon)^2}{(1-\varepsilon)^4} \mu_{e_i}(B_{e}(z,r))\leq 4\widetilde c \left(\frac{1+\varepsilon}{1-\varepsilon}\right)^4 \mu_e(B_e(z,r))\\
&=4\widetilde c \left(\frac{1+\varepsilon}{1-\varepsilon}\right)^4 \mu(B(z,r)),
\end{align*}
which setting 
$$
U:=B\Big(x, \frac{\rho}{2(1+\varepsilon)}\Big),
$$
completes the proof.
\end{proof}

Using that $(S,d)$ is a complete locally compact length space, standard arguments entail that the doubling estimate holds on arbitrary relatively compact open subsets and that the weak Poincaré estimate also holds on arbitrary relatively compact open subsets. A combination of these two facts implies that the weak Poincaré estimate self-improves to a strong Poincaré estimate by \cite{sturmAnalysisLocalDirichlet1996}. We summarize these results in :

\begin{corollary}\emph{(i)} For every open relatively compact $U\subset S$ there exists a constant $C>0$, depending on $U$, such that for any $f\in W^{1,2}(S,d,\mu)$ and any $z\in U$ with $B(z,r)\subset U$ one has 
\[
\int_{B(z, r)}|f-f_B|^2\dd\mu \leq Cr^2 \int_{B(z,r)} |Df|^2\dd\mu,\quad \text{where $f_B := \fint_{B(z,r)}f\dd\mu$.}
\]
\emph{(ii)} For every open relatively compact $U\subset S$ there exists a constant $c>0$, depending on $U$, such that for any $z\in U$ and every $r>0$ with $B(z,2r)\subset U$ one has
\[
\mu(B(z, 2r)) \leq c \mu(B(z,r)).
\]
\end{corollary}

We recall that, as the Cheeger energy is a strongly local regular Dirichlet form and we get the associated Cheeger Laplacian $\Delta$ as a nonpositive self-adjoint operator in $L^2(S,\mu)$. The heat semigroup $\ee^{t\Delta}$, $t>0$, is then defined a priori in terms of the spectral calculus. Being equipped with \thref{thm:AnalysisOnBIC} and Sturm's results on analysis on strongly local regular Dirichlet spaces, we immediately obtain: 

\begin{corollary}
\th\label{cor:ExistenceHeat}
There exists a uniquely determined Hölder continuous function 
$$
p: (0,\infty) \times S \times S \longrightarrow  (0,\infty)
$$
the \emph{heat kernel of $(S,d)$}, such that for any function $f\in L^2(S,d,\mu)$ and $\mu$-a.e. $x\in S$,
\[
\ee^{t\Delta}f(x) = \int_S p(t,x,y)f(y)\dd\mu(y).
\]
Moreover, for all $x,y\in S$, $s,t>0$ one has
\[
p(t,x,y) = p(t,y,x),\quad p(s+t,x,y) = \int_S p(t,x,z)p(s,z,y)\dd\mu(z).
\]
\end{corollary}

\begin{proof}
Using that $d_{\Ch}$ induces the original topology of $S$, and that one has local doubling and local Poincaré estimates, the results of \cites{sturmAnalysisLocal1995, sturmAnalysisLocalDirichlet1996} entail the existence of a nonnegative Hölder continuous heat kernel. The proof of \cite[Theorem 1.5(2d)]{kuwaeSobolevSpacesLaplacian2001} applies to show positivity.
\end{proof}

For the formulation of the following result, note that for every $f\in \LIP_\loc(S,d)$ the exterior derivative $df$ exists $\mu$ a.e. Indeed, pick a smooth Riemannian metric $h$ on $S$ and $u\in\mathcal{V}(S,h)$ such that $d=d_{h,u}$. Consider a relatively compact chart $V\subset S$. As $(S,d)$ does not contain any cusp, one can shrink $V$ so that $e^u\in L^p(V)$ for some $p>2$ by Brézis-Merle's Theorem, see \cite{brezisUniformEstimates1991}. Now, for any absolutely continuous curve $\gamma:[0,1]\to V$ one has
\[
|f(\gamma(1))-f(\gamma(0))| \leq \LIP(f)\,d(\gamma(1),\gamma(0)) \leq \LIP(f)\int_\gamma e^u
\]
by \cite[Lemma 4.72]{fillastreSubharmonic2023}. Thus we get $f\in W^{1,p}(V)$ by \cite{gigliDifferentialStructureMetric2015}. Finally, by Calder\'on's Theorem, $f$ is differentiable Lebesgue a.e. in $V$ and so $df$ is well-defined $\mu$-a.e. in $S$.

\begin{proposition}\th\label{thm:DirichletEqCheeger}
\emph{(i)} Pick a smooth Riemannian metric $h$ on $S$ and $u\in\mathcal{V}(S,h)$ such that $d=d_{h,u}$. Then for any $f\in \LIP_c(S,d)$ one has $|Df|=|df|_{h,u}:=\ee^{-u}|d f|_h$ $\mu$-a.e. in $S$. \\
\emph{(ii)} $C^\infty_c(S)$ is dense in $W^{1,2}(S,d,\mu)$.
\end{proposition}

\begin{proof} (i) As $f$ is compactly supported, in view of the localization argument of \thref{thm:AnalysisOnBIC}, we can assume that $S$ is diffeomorphic to $\mathbb{S}^2$ and that there exists a sequence $(d_i)$ of polyhedral distances on $S$ such that $d_i\to d$ in the Lipschitz sense. We have $d_i=d_{h_i, u_i}$ with $h_i$ a smooth Riemannian metric on $S$ and $u_i\in\mathcal{V}(S,h_i)$. However, as $S$ is diffeomorphic to $\mathbb{S}^2$, each $h_i$ can be written $h_i = \ee^{2v_i}h$ with $v_i$ a smooth function on $S$. In view of $u_i,v_i\in\mathcal{V}(S,h)$, we can apply \cite[Theorem 4.13]{chenUniformConvergenceMetrics2025} and by Rellich-Kondrachov, $u_i+v_i\to u$ in $L^q(S,\mu_h)$ for some $q>1$. After passing to a subsequence, we can assume that $u_i+v_i\to u$ and so 
$$
|df|_{h,u_i+v_i}\to |df|_{h,u}\quad\text{$\mu$-a.e. .}
$$
Since the $(S,d_{h,u_i+v_i})$ are polyhedral surfaces, we can pick a polyhedrally smooth Riemannian metric $g_i$ on $(S,d_{h,u_i+v_i})$ such that $d_{h,u_i+v_i}=d_{g_i}$, and so 
$$
|Df|_{d_{h,u_i+v_i}} = |df|_{g_i}\quad\text{$\mu$-a.e.} 
$$
by Remark \ref{rypo}, where $|df|_{g_i}$ is understood in the Riemannian polyhedral sense. Similarly to the proof of Theorem \ref{letsa} one finds 
$$
|df|_{g_i}=|df|_{h,u_i+v_i}\quad\text{$\mu$-a.e. .}
$$
On the other hand, given $\varepsilon >0$, for all large enough $i$ we get
 \[
 \frac{1}{1+\varepsilon} |Df| \leq |Df|_{d_{h,u_i+v_i}} \leq \frac{1}{1-\varepsilon} |Df|\quad\text{ $\mu$-a.e. .}
 \]
Thus,  
$$
|Df|_{d_{h,u_i+v_i}} \to |Df|\quad\text{$\mu$-a.e.,}
$$
and in summary 
$$
|Df|=|df|_{h,u}\quad\text{$\mu$-a.e.,}
$$
completing the proof.\\
(ii) Pick a smooth Riemannian metric $h$ on $S$ and $u\in\mathcal{V}(S,h)$ such that $d=d_{h,u}$. Let $f\in\LIP_c(S,d)$. We can assume w.l.o.g. that $f$ is compactly contained in a chart $V\subset\mathbb{R}^2$. Let $\phi$ be a Friedrichs mollifier on $\mathbb{R}^2$. For the zeroth order term of the Sobolev norm we have
$$
\int_V |f\ast \phi_{a_n}-f|^2\dd\mu=\int_V |f\ast \phi_{a_n}-f|^2\ee^{2u}\dd\mu_h \xrightarrow[n\to \infty]{} 0
$$
for some sequence $a_n\to 0$ by dominated convergence, because $f\ast \phi_{1/n }\to f$ in $L^1_{\text{loc}}(V,\mu_h)$ (and thus $\mu_h$-a.e. for some subsequence $a_n$ of $1/n$) and because the conformal factor $\ee^{2u}$ is in $L^1_\loc(V,\mu_h)$. For the first-order term, by part (i) we have
\[
\int_V |D(f\ast \phi_\varepsilon - f)|^2\dd\mu
=\int_V |d(f\ast \phi_\varepsilon - f)|_{h,u}^2\dd \mu_{h,u}= \int_V |d(f\ast \phi_\varepsilon - f)|_{h}^2\dd \mu_h \xrightarrow[\varepsilon \to 0]{} 0,
\]
e.g. by \cite[Proposition I.22]{guneysuCovariantSchrodingerSemigroups2017}.
\end{proof}

\begin{remark}Picking a Riemannian metric $h$ on $S$ and $u\in\mathcal{V}(S,h)$ such that $d=d_{h,u}$, and denoting the Cheeger-Laplacian with $\Delta =\Delta_{h,u}$, one has
$$
\Delta_{h,u}f=\ee^{-2u} \Delta_hf\quad\text{for all $f \in C^\infty_c(S)$.}
$$
This follows from 
\begin{align*}
 -\int (\Delta_{h,u} f) f \dd\mu_{h,u} &=\int |Df|^2 \dd\mu=\int |df|^2_{h,u} \dd\mu_{h,u}=\int |df|^2_{h} \dd\mu_{h}= -\int (\Delta_{h} f) f \dd\mu_{h}\\
 &=-\int (\ee^{-2u}\Delta_{h} f) f \dd\mu_{h,u}.
\end{align*}
\end{remark}

\section{BIC surfaces with curvature in Dynkin class}
\label{section:dynkin}

Having established these preliminary results, we now introduce BIC-Dynkin surfaces.\vspace{2mm} 

In this section, let $(S,d)$ be a complete BIC surface without cusps, let $\mu$ be the induced $2$-dimensional Hausdorff measure, and let $\omega$ denote its curvature measure. Furthermore, $\Delta$ denotes the Cheeger Laplacian in $L^2(S,\mu)$ and $p(t,x,y)>0$ the heat kernel.
\vspace{2mm}

We recall that the Choquet capacity with respect to the Cheeger energy is given on $U\subset S$ open by 
$$
\CAP(U):=\inf \{ \>\|f\|_{W^{1,2}}\>\>|\>\>\text{$f\geq 1$ on $U$}\}
$$
and extended to arbitrary Borel sets by outer regularity. A function $\nu$ from the Borel sets of $S$ to $[0,\infty]$ is called \emph{capacity continuous} if $\nu(N) = 0$ for all Borel sets $N \subset S$ with $\CAP(N)=0$.

\begin{definition}
The \textit{Dynkin class} $\Dyn(S,d)$ is defined to be the set of nonnegative capacity continuous Radon measures $\nu$ on $S$ such that 
\begin{align}\label{ewq}
\sup_{x\in S}\int^{t_0}_0\int_S p(s,x,y) \dd\nu(y) \dd s <\infty\quad\text{for some $t_0>0$}.
\end{align}
\end{definition}

In the literature, $\nu$'s with 
$$
\sup_{x\in S}\int^{t_0}_0\int_S p_s(x,y) \dd\nu(y) \dd s <1\quad\text{for some $t_0>0$}
$$
are sometimes called \emph{contractive Dynkin measures}, and $\nu$'s with 
$$
\lim_{t\to 0+}\sup_{x\in S}\int^{t}_0\int_S p_s(x,y) \dd\nu(y) \dd s =0
$$
are called \emph{Kato measures}. Clearly (\ref{ewq}) is equivalent to 
$$
\sup\left\{ \left|\int_S\int^{t_0}_0 \ee^{s\Delta} f(x)\dd s\dd\nu(x)\right| \>\>:\>\> f \in C^\infty_c(S), \int_S |f|\dd\mu \leq 1 \right\} < \infty\>\>\text{for some/all $t_0>0$},
$$
which is, in turn, by \cite{kuwae}, equivalent to
$$
\sup\left\{ \left|\int_S (-\Delta+\lambda_0)^{-1} f(x)\, \dd\nu(x)\right| \>\>:\>\> f \in C^\infty_c(S), \int_S |f|\dd\mu \leq 1 \right\} < \infty\quad\text{for some/all $\lambda_0>0$},
$$
and in its latter form the Dynkin class has been introduced for general regular Dirichlet forms in \cite{stollmannPerturbationDirichletForms1996}.\vspace{1mm}

The Dynkin class encompasses a wide range of measures. Here are some simple ones (see the example below for a geometric example):

\begin{example}1. If $\nu$ is a nonnegative Radon measure such that $\nu\leq c\mu$ for some constant $c>0$, then $\nu$ is capacity continuous and for every $\lambda_0>0$ and $f\in C^\infty_c(S)$ there holds
\[
\left| \int_S (-\Delta+\lambda_0)^{-1}f(x)\dd\nu(x) \right|\leq (c/\lambda_0)\int_S |f(x)|\dd\mu(x),
\]
as $(-\Delta+\lambda_0)^{-1}$ is a Markovian operator, showing that $\nu\in \Dyn(S,d)$.\vspace{1mm}

2. By the last example, $\omega^-\in\Dyn(S,d)$, whenever $(S,d)$ is $\BICB(\kappa)$ for some $\kappa\in\R$.

\end{example}

\begin{definition} $(S,d)$ is called a \emph{BIC-Dynkin} surface if $\omega^-\in\Dyn(S,d)$.
\end{definition}

Note that, by our standing assumptions, BIC-Dynkin surfaces are complete and without cusps.

\begin{theorem}
\th\label{thm:BICDynkinDeformation}
Assume that $(S,d)$ is a BIC-Dynkin surface. Pick a smooth Riemannian metric $h$ on $S$ and $u\in\mathcal{V}(S,h)$ such that $d=d_{h,u}$.
Then there exist a constant $\kappa\in\R$ and a function $\psi\in L^\infty(S,\mu_{h,u}) \cap\mathcal{V}(S,h)$,
such that $(S, d_{h,u-\psi})$ is a $\BICB(\kappa)$ space, in particular, $(S,d)$ is bi-Lipschitz equivalent to a surface in $\BICB(\kappa)$.
\end{theorem}

\begin{proof} As we have $\omega^-_{h,u}=\omega^-\in\Dyn(S,d)=\Dyn(S,d_{h,u})$, the nonnegative linear functional 
\[
C_c^\infty(S)\ni f\longmapsto \int_S  (1-\Delta_{h,u})^{-1}f\dd \omega_{h,u}^-
\]
extends to a nonnegative bounded linear functional on $L^1(S,\mu_{h,u})$, which is represented by some function $0 \leq \psi \in L^\infty(S,\mu_{h,u})$. For all $\varphi\in C^\infty_c(S)$ we have
\begin{align*}
\int \varphi\dd \omega_{h,u}^-
&= \int (1-\Delta_{h,u})^{-1}(1-\Delta_{h,u})\varphi\dd \omega_{h,u}^- \\
&= \int \psi(1-\Delta_{h,u})\varphi\dd\mu_{h,u},
\end{align*}
and so
\begin{align*}
\int \psi(\Delta_h\varphi)\dd\mu_h
&= \int \psi \ee^{2u}(\Delta_h\varphi)\ee^{-2u}\dd\mu_h \\
&= \int \psi (\Delta_{h,u}\varphi)\dd\mu_{h,u} \\
&= \int \varphi\dd\big(\psi\mu_{h,u} - \omega_{h,u}^-\big)
\end{align*}
leading to $\underline{\Delta_h}\psi = \psi\mu_{h,u} - \omega_{h,u}^-$. Finally, if $\varphi\geq 0$, we have
\begin{align*}
\int \varphi\dd\omega_{h,u-\psi}
&= \int \varphi\,\mathbb{K}_h\dd\mu_h + \int \varphi\dd [\underline{\Delta_h}(u-\psi)] \\
&= \int \varphi\,\mathbb{K}_h\dd\mu_h + \int \varphi\dd [\underline{\Delta_h} u] - \int \varphi\dd[\underline{\Delta_h} \psi] \\
&= \int \varphi\dd\omega_{h,u}^+ - \int \varphi\dd\omega_{h,u}^-
 - \int \varphi\psi\dd\mu_{h,u} + \int \varphi\dd\omega_{h,u}^- \\
&= \int \varphi\dd\omega_{h,u}^+ - \int \varphi\psi\dd\mu_{h,u} \\
&\geq - \int \varphi \psi\dd\mu_{h,u} \\
&\geq - (\sup\psi) \ee^{- 2(\inf\psi)} \int \varphi\dd\mu_{h,u-\psi},
\end{align*}
completing the proof.
\end{proof}

As a consequence, we get the following global results in the BIC-Dynkin case:

\begin{corollary} Assume $(S,d)$ is a BIC-Dynkin surface.\\
\emph{(i)} There exists $c>0$ such that for $0<r<r'$ and $x\in S$,
\[
\mu(B(x,r'))\leq c \ee^{c r'} \left( \frac{r'}{r} \right)^2 \mu(B(x,r)).
\]
\emph{(ii)} There exist $C>0$ such that for all $x\in S$, $r>0$, $f\in W^{1,2}(S,d,\mu)$ one has 
\[
\int_{B(x, r)}|f-f_B|^2\dd\mu \leq C \ee^{Cr} r^2 \int_{B(x,r)} |Df|^2\dd\mu,\quad\text{where $f_B := \fint_{B(x,r)}f\dd\mu$.}
\]
\emph{(iii)} There exist constants $c_i>0$, $i=1,\dots,6$, such that, for all $x,y\in S$ and $t>0$,
\[
\frac{c_1\ee^{-c_2t}}{\mu(B(x,\sqrt{t}))} \ee^{-\frac{d(x,y)^2}{c_3 t}}
\leq p(t,x,y) \leq  \frac{c_4\ee^{c_5t}}{\mu(B(x,\sqrt{t}))} \ee^{-\frac{d(x,y)^2}{c_6 t}}.
\]
\end{corollary}

\begin{proof} Note first that these statements are all stable under a bi-Lipschitz homeomorphism, see \cite{sturmAnalysisLocalDirichlet1996}. By Theorem \ref{thm:BICDynkinDeformation} the surface $(S,d)$ is bi-Lipschitz equivalent to a complete $\BICB(\kappa)$ surface without cusps, which by Remark \ref{edpopo} is an $\RCD(\kappa, 2)$ space and, as such, satisfies all the statements of the theorem: Doubling is proved in \cite{sturm2} and the Poincaré inequality is proved in its $(1,1)$ form in  \cite{rajalaLocalPoincare2012} and implies the stated (2,2) variant in a weak form by standard arguments \cite{hajlaszSobolevPoincare1995}, and then self-improves to its strong form again by standard arguments \cite{sturmAnalysisLocalDirichlet1996}; the Gaussian bound comes from \cite{jiangHeatKernelBound2016}.
\end{proof}

We close this section with an illustrative example of a BIC-Dynkin surface that is neither $\BICB(\kappa)$ nor $\CBB(\kappa,2)$:

\begin{example}(Volcano of angle $4\pi$) Consider $\R^2$ endowed with the subharmonic metric $g:=\ee^{2\varphi}h$, with $h$ the Euclidean metric and with curvature measure $\omega := -\mathcal{H}^1|_{\mathbb{S}^1}$,
the one-dimensional Hausdorff measure concentrated on the circle. Note that $\varphi$ is given by
\begin{align*}
\varphi(z)
&= -\frac{1}{2\pi} \int_{\R^2} \ln|z-\zeta|\dd\omega(\zeta) \\
&= \frac{1}{2\pi} \int_{\mathbb{S}^1} \ln|z-\zeta|\dd\mathcal{H}^1(\zeta) \\
&= \frac{1}{2\pi} \int_0^{2\pi} \ln|z-\ee^{i\theta}|\dd\theta \\
&=
\begin{cases}
0, & |z|\leq 1, \\
\ln|z|, & |z|\geq 1,
\end{cases} \\
&= \ln(\max(|z|,1)).
\end{align*}
The last step is obtained by a direct application of Jensen's formula. The subharmonic metric is explicitly given by
\[
g =
\begin{cases}
|\dd z|^2, & |z|\leq 1, \\
|z|^2|\dd z|^2, & |z|\geq 1.
\end{cases}
\]
By writing $g$ in polar coordinates,
\[
g = \rho^2 d\rho^2 + \rho^4d\theta^2,
\]
and applying the variable change $r=\frac{1}{2}\rho^2$,
\[
g = dr^2 + 4 r^2d\theta^2,
\]
when $|z|\geq 1$. Geometrically, this shows that $(\R^2,d_{\varphi, h})$ is a disc with a cone of angle $4\pi$ glued along $\mathbb{S}^1$.
This explains the name of \emph{volcano of angle $4\pi$}. By \cite[Theorem 13.3 Chapter 8]{fillastreSubharmonic2023}, $d_{\varphi,h}$ induces the Euclidean topology coming from $d_h$. Note that in view of $\varphi\geq 0$ we get that $d_{\varphi,h}$ is a complete distance.\\
The conformal factor $\ee^{2\varphi(z)}=\max(|z|^2, 1)$ is a strong $A_\infty$-Muckenhoupt weight. Indeed, the conformal factor can be realized as the determinant of the Jacobian of the quasiconformal map $f:\R^2\to\R^2$ defined by
\[
f(z) = \phi(|z|) \frac{z}{|z|}, \quad
\phi(r) =
\begin{cases}
r, & 0\leq r\leq 1, \\
\sqrt{\frac{r^4+1}{2}}, & r\geq 1.
\end{cases}
\]
Recall that, for a plane radial map, the determinant of the Jacobian is
\[
\det J_f(z) = \frac{\phi(|z|)\phi'(|z|)}{|z|}.
\]
If $|z|\leq 1$, $\det J_f(z)=1$.
If $|z|\geq 1$, $\phi'(r)=r^3/\phi(r)$ and so $\det J_f(z) = |z|^2$.
Hence $\det J_f(z) = \max(|z|^2,1)$.
Moreover, it is known (see \cite[Part 2.6]{astalaEllipticPartial2009}) that, for a plane radial map, the quasiconformal dilation coefficient is given by
\[
K_f(x) = \max\left( \frac{|z|\phi'(z)}{\phi(z)}, \frac{\phi(z)}{|z|\phi'(z)} \right).
\]
So, by direct computation, $f$ is $2$-quasiconformal.
We deduce that $(\R^2, d_{\varphi,h})$ is $2$-Ahlfors regular and admits a scale-invariant uniform $(1,1)$-Poincaré inequality on balls. See \cite[Proposition 6.10]{kinnunenRegularitySets2013} and \cite{heinonenQuasiconformalMaps1998} for the strong $A_{\infty}$-weight criterion.
Therefore, by \cite{sturmAnalysisLocal1995}, the heat kernel of $(\R^2, d_{\varphi,h})$ has a global Gaussian upper bound. We can conclude that $\omega^-$ is in $\Dyn(\R^2,d_{\varphi,h} )$ since $\omega^-\in\Dyn(\R^2, d_h)$
by \cite[Corollary 2.30]{erbarTamedSpacesDirichlet2020}. Because $\omega^-$ is not absolutely continuous with respect to $\mu_{\varphi,h}$, it is clear that $(\R^2, d_{\varphi,h})$ is neither $\BICB$ nor $\CBB$. Note that by \cite[Theorem 6.9]{heinonenQuasiconformalMaps1998} the metric measure space $(\R^2, d_{\varphi,h}, \mu_{\varphi,h})$ is even a Loewner space.
\end{example}

\appendix

\section{Riemannian Polyhedra}
\label{section:RiemPolyhedra}

In this appendix, we recall some basic definitions and results of the theory of Riemannian polyhedra, following the book \cite{eellsHarmonicMapsRiemannian2001} by Eells and Fuglede.

\subsection{Locally Lipschitz polyhedra}

A \textit{(countable and locally finite) simplicial complex} $K$ consists of a countable set $V$, called \textit{vertices of $K$},
and a set $S$ of finite non-empty subsets of $V$, called \textit{the simplexes of $K$}, satisfying the following properties:
\begin{itemize}
\item every singleton belongs to $S$,
\item every subset of a simplex is a simplex,
\item every vertex belongs to a finite number of simplexes.
\end{itemize}
A simplex $a\in S$ is said to be a \textit{face} of a simplex $b\in S$ if $a\subset b$.\vspace{1mm}

A \textit{$p$-simplex} or a \emph{$p$-dimensional simplex} is a simplex with exactly $p+1$ vertices. \vspace{1mm}

We denote by $\mathcal{S}^{(p)}(K)$ the set of $p$-simplexes of $K$ and by $\mathcal{V}(K):=V$ and
$$
\mathcal{S}(K):=\bigcup_{p}\mathcal{S}^{(p)}(K)
$$

A simplex is called \textit{maximal} if it is not contained in another simplex and a simplicial complex is called \emph{homogeneous of dimension $m$} if all maximal simplexes have the same dimension $m$. \vspace{1mm}

The space $|K|$ is the set of all formal finite linear combinations of $$
\alpha = \sum_{v\in \mathcal{V}(K)}\alpha(v)v
$$
such that 
$$
\alpha(v)\geq 0, \quad \sum_{v\in \mathcal{V}(K)}\alpha(v)=1,\quad \{ v\in \mathcal{V}(K) | \alpha(v)>0 \}\in \mathcal{S}(K).
$$
Hence $|K|$ is a subset of the linear space $\langle K\rangle $ of formal finite linear combinations 
$$
\alpha = \sum_{v\in \mathcal{V}(K)}\alpha(v)v.
$$

One puts a distance on $|K|$, called \textit{barycentric distance}, by setting
\[
d_{\mathrm{bar}}(\alpha, \beta)^2 := \sum_{v} (\alpha(v) - \beta(v))^2,
\]
where $\alpha=\sum\alpha(v)v$ and $\beta=\sum\beta(v)v$. \vspace{1mm}

By the usual abuse of notation/language, the closure of a simplex $s$ is defined by
$$
\overline{s}:=\Big\{\alpha\in |K|\>:\> \{v|\alpha(v)>0\}\subset s\Big\},
$$
and the interior of a simplex $s$ is defined by 
$$
s^\circ:=\Big\{\alpha\in |K|\>:\> \{v|\alpha(v)>0\}=s\Big\}\subset \overline{s}.
$$
For every $s\in \mathcal{S}^{(p)}(K)$ there exists an affine bijection from $\overline{s}$ onto a subset $\R^p$ given by the convex hull of $p+1$ affinely independent vectors. Any such map is called an \emph{affine coordinate system} for $s$.\vspace{1mm} 

A connected locally compact Hausdorff topological space $X$ is called a \textit{polyhedron} if there exists a complex $K$ and a homeomorphism 
$$
\vartheta : |K|\longrightarrow  X.
$$
Any such pair $T=(K,\vartheta)$ is called a \textit{triangulation} of $X$ and one sets $\mathcal{S}^{(p)}(X,T):=\mathcal{S}^{(p)}(K)$.\vspace{1mm}

A polyhedron $X$ is called \emph{homogeneous of dimension $m$} if some (and then every) triangulation of $X$ has this property. \vspace{1mm}

We add the following simple definition for future reference:

\begin{definition}\label{lbldef} A set together with an equivalence class of locally bi-Lipschitz distance functions will be called an \emph{LBL (local bi-Lip) space}.
\end{definition}

Note that every LBL space carries a canonic Hausdorff topology. A connected locally compact LBL space $X$ is called a \textit{locally Lipschitz polyhedron} if there exists a triangulation $T=(K,\vartheta)$
such that 
$$
\vartheta : |K|\longrightarrow  X
$$
is a locally Lipschitz map. In this case, $T$ is also-called a \textit{locally Lipschitz triangulation of $X$}. For example, any polyhedron becomes a locally Lipschitz polyhedron if $X$ is endowed with the pushforward of the barycentric metric by $\vartheta$.

\subsection{Polyhedral Riemannian metrics}

A \textit{Riemannian polyhedron} is defined as a homogeneous locally Lipschitz polyhedron $X$ of dimension $m$ together with an assignment which to every locally Lipschitz triangulation $T=(K,\vartheta)$ of $X$, every $s\in \mathcal{S}^{(m)}(X,T)$ and every affine coordinate system $\phi$ for $s$ assigns a measurable and pointwise positive definite function 
$$
g^T_s: s^\circ\longrightarrow \mathrm{Mat}(m\times m;\R),
$$
which in some (and then every) affine coordinate system $\phi$ for $s$ satisfies the following ellipticity condition:
\begin{align*}
&\textit{there exists a constant $\Lambda_s>0$ such that for a.e. $v\in \phi(s^\circ)$ and all $\xi\in\mathbb{R}^m$ one has}\\
&\quad\quad\quad\Lambda_s^{-2} \sum_{i=1}^m(\xi^i)^2
\leq \sum_{1 \leq i,j \leq m} g^T_{s,ij}(\phi^{-1}(v))\xi^i\xi^j
\leq \Lambda_s^{2} \sum_{i=1}^m(\xi^i)^2,
\end{align*}
and which is such that if $T=(K,\vartheta)$ and $\underline{T}=(\underline{K},\underline{\vartheta})$ are locally Lipschitz triangulations of $(X,d)$ and $s\in \mathcal{S}^{(m)}(X,T)$ and $\underline{s}\in \mathcal{S}^{(m)}(X,\underline{T})$ are such that $\vartheta(s^\circ) \cap \underline{\vartheta}(\underline{s}^\circ)\neq \emptyset$, then for every affine coordinate system $\phi$ for $s$ and $\underline{\phi}$ for $\underline{s}$, one has the covariant transformation rule
$$
g^T_{s}(\phi^{-1}(v))=\tau'(v)^{\dagger}\,g^{\underline{T}}_{\underline{s}}(\underline{\phi}^{-1}(\tau(v)))\,\tau'(v)\quad\text{for a.e. $v$, where $\tau:=\underline{\phi}\circ\underline{\vartheta}^{-1}\circ \vartheta\circ\phi^{-1}$}.
$$

Then $g$ is called a \emph{polyhedral Riemannian metric} on $(X,d)$.\vspace{1mm}

Note that if for some locally Lipschitz triangulation $T$ of $X$ and every $s\in \mathcal{S}^{(m)}(X,T)$ one is given a measurable Riemannian metric on $s^\circ$ which satisfies ellipticity, then one canonically gets a polyhedral Riemannian metric by requiring the covariant transformation rule to hold.\vspace{3mm}

Let $(X,g)$ be a Riemannian polyhedron of dimension $m$. We fix a locally Lipschitz triangulation $T=(K,\vartheta)$ of $X$.\vspace{3mm}

If $s\in\mathcal{S}^{(m)}(X,T)$, then for any Borel set $E\subset s^\circ$ one defines
\[
 \mu_{g^T_s}(E) := \int_{\phi(E)} \sqrt{\det g^T_s(\phi^{-1}(v))}dv,
 \]
where $\phi$ is any affine coordinate system for $s$ (this is well-defined). Then, for any Borel set $A\subset X$ one defines
\[
\mu_g(A) := \sum_{s\in\mathcal{S}^{(m)}(X,T)} \mu_{g^T_s}(\vartheta^{-1}(A)\cap s^\circ),
\]
This construction does not depend on $T$ and leads to \emph{Riemannian volume measure $\mu_g$ on $(X,d,g)$}, a fully supported Radon measure. Note that the sets $\vartheta(s^\circ)$, $s\in\mathcal{S}^{(m)}(X,T)$, are pairwise disjoint with 
$$
\mu_g\Big(X\setminus \bigcup_{s\in\mathcal{S}^{(m)}(X,T) }\vartheta(s^\circ)\Big)=0.
$$

Let $Z_{g,T}\subset X$ be the union of non-Lebesgue points of $g^T_s$ over $s\in\mathcal{S}^{(m)}(X,T)$, which is a $\mu_g$-null-set. For every couple of points $x,y\in X$ and $\mu_g$-null-set $Z\subset X$ such that $Z\supset Z_{g,T}$, we define $\Omega(x,y;Z)$ to be
the set of locally Lipschitz paths $\gamma: [0,1] \to X$
such that $\gamma(0)=x, \gamma(1)=y$ and $\gamma^{-1}(Z)$ is a $\mu_g$-null-set. For any path $\gamma\in\Omega(x,y;Z)$ its polyhedral $g$-length is defined by
\[
\mathscr{L}_g(\gamma) := \sum_{s\in S^{(m)}(X,T)} \bigintsss_{\gamma^{-1}(\vartheta(s^\circ))}
\sqrt{\sum_{1\leq i,j\leq m}g^T_{s,ij}\big( \vartheta^{-1} \circ \gamma(t)\big) \dot\gamma^i_{\phi_s}(t)\dot\gamma^j_{\phi_s}(t)}\dd t,
\]
where 
$$
(\gamma^1_{\phi_s},\dots,\gamma^m_{\phi_s})=\phi_s\circ \vartheta^{-1} \circ \gamma,
$$
with $\phi_s$ an affine coordinate system with respect to $s$ (well-defined). We write
\[
d^Z_g(x,y) := \inf_{\gamma\in\Omega(x,y;Z)}\mathscr{L}_g(\gamma).
\]
Clearly, $Z\supset Z' \supset Z_{g,T}$ implies $d^Z_g(x,y) \geq d^{Z'}_g(x,y)$.
Finally, we introduce the \textit{Riemannian distance induced by $g$}, denoted with $d_g$, via
\[
d_g(x,y) := \sup\{d^Z_g(x,y) \mid Z\supset Z_{g,T}, Z~\text{ is a $\mu_g$-null-set} \}
\]
By \cite[Proposition 4.1]{eellsHarmonicMapsRiemannian2001}, $d_g$ is a (finite) distance on $X$, which is independent of $T$, and $(X,d_g)$ is a length space. By the above ellipticity assumption, $d_g$ induces the original bi-Lip structure on $X$.\vspace{2mm}

We set
$$
\LIP(X,g):=\LIP(X,d_g),\quad \LIP_{\loc}(X,g):=\LIP_{\loc}(X,d_g).
$$
If $f\in\LIP_{\loc}(X,g)$, then for every $s\in\mathcal{S}^{(m)}(X,T)$ and every affine coordinate system $\phi$ for $s$ the function $f\circ \vartheta\circ \phi^{-1}$ is locally Lipschitz on $\phi(s^\circ)$ in the Euclidean sense. Thus, by Rademacher's theorem, $\partial_ j(f\circ \vartheta\circ \phi^{-1})$ exists and agrees with the distributional derivative $\partial^{\mathrm{distr}}_ j(f\circ \vartheta \circ \phi^{-1})$ a.e.\ in $\phi (s^\circ)$, and we get a $\mu_g$-a.e.\ well-defined Borel function 
$$
|df|_g: X\longrightarrow [0,\infty)
$$
by
$$
|df|_g^2:=\sum^m_{i,j=1}\Big( (g_T^{s,ij}\circ \phi^{-1})\cdot \partial_i(f\circ \vartheta \circ \phi^{-1})\cdot\partial_j(f\circ \vartheta \circ \phi^{-1})\Big)\circ \phi\circ \vartheta^{-1}\quad\text{ in $\vartheta(s^\circ)$}.
$$
We denote by $\LIP_{\loc}^{1,2}(X,g)$ the space of all $f\in \LIP_{\loc}(X,g)$
for which the following Sobolev norm is finite:
\[
\lVert f \rVert_{W_{\text{EF}}^{1,2}(X,g)}^2 = \int_X (f^2 + |df|_g^2)\dd\mu_g. 
\]
The Sobolev space $W_{\text{EF}}^{1,2}(X,g)$ is defined as the completion of $\LIP_{\loc}^{1,2}(X,g)$ by the norm $\lVert\cdot\rVert_{W_{\text{EF}}^{1,2}(X,g)}$.
Here, the letters EF in the subscript stand for Eells-Fuglede to not confuse it with Sobolev spaces of Section \ref{section:SobolevSpaces}. This is a dense subspace of 
$$
L^2(X,g):=L^2(X,\mu_g).
$$

If $f\in W^{1,2}_{\mathrm{EF}}(X,g)$, then for every $ s\in\mathcal{S}^{(m)}(X,T)$ and every affine coordinate system $\phi$ for $s$ the function $f\circ \vartheta\circ \phi^{-1}$ is in the Euclidean Sobolev space $W^{1,2}(\phi(s^\circ))$, so 
$$
\partial^{\mathrm{distr}}_ j(f\circ \vartheta \circ \phi^{-1})\in L^2(\phi(s^\circ)),
$$
and we get a $\mu_g$-a.e. well-defined Borel function 
$$
|df|_g: X\longrightarrow [0,\infty),
$$
by setting
$$
|df|_g^2:=\sum^m_{i,j=1}\Big( (g_T^{s,ij}\circ \phi^{-1})\cdot \partial^{\mathrm{distr}}_i(f\circ \vartheta \circ \phi^{-1})\cdot\partial^{\mathrm{distr}}_j(f\circ \vartheta\circ \phi^{-1})\Big)\circ \phi \circ \vartheta^{-1}\quad\text{in $\vartheta(s^\circ)$}.
$$
One has the following locality rule: If $f_0,f_1\in W^{1,2}_{\mathrm{EF}}(X,g)$ and $f_0=f_1$ $\mu_g$-a.e. on some open $Y\subset X$, then $|df_0|_g=|df_1|_g$ $\mu_g$-a.e. in $Y$, which follows from the corresponding locality result in $W^{1,2}(s^\circ)$.\vspace{2mm}

The \textit{local Sobolev space} $W^{1,2}_{\mathrm{EF},\loc}(X,g)$ is defined as the set of all $f\in L^2_{\loc}(X,g)$ such that $f|_V\in W^{1,2}_{\mathrm{EF}}(V,g)$ for every relatively compact open and connected $V\subset X$. By the locality rule, it is clear $|d u |_g$ is defined for every $W^{1,2}_{\mathrm{EF},\loc}(X,g)$, too.\vspace{1mm}

We remark that none of the above definitions depends on $T$. Also, by Whitehead's Theorem, every smooth Riemannian manifold is (consistently) a Riemannian polyhedron.\vspace{2mm}

From now on we will understand all data with respect to the fixed polyhedral metric $g$.

\begin{theorem}
\th\label{riemPolyhedra:th:analyticProps}
\emph{(i)} $W_{\mathrm{EF}}^{1,2}(X)$ is a Hilbert space.\\
\emph{(ii)} For open relatively compact $U\subset X$ there exists a constant $c>0$, depending on $U$, such that for all $z\in U$ and all $r>0$ with $B(z,2 r)\subset U$,
\[
\mu(B(z,2 r)) \leq c  \mu(B(z,r)).
\]
\emph{(iii)} For every open relatively compact $U\subset X$ there exist constants $C>0$ and $\kappa>0$, depending on $U$, such that for all $z\in U$ all $r>0$ with $B_g(z,\kappa r)\subset U$ and all $f\in W_{\mathrm{EF}}^{1,2}(X)$,
\[
\int_{B(z,r)} |f-f_{B}|^2\dd\mu \leq C r^2 \int_{B(z,\kappa r)} |df|^2\dd\mu,\quad \text{where $f_{B} := \fint_{B(z,r)}f\dd\mu$}.
\]
\emph{(iv)} For all $x,y\in X$ one has
\begin{align*}
d(x,y) &= \sup\{ f(x)-f(y) \mid f\in\LIP(X), |d f |\leq 1 \} \\
&= \sup\{ f(x)-f(y) \mid u\in W^{1,2}_{\mathrm{EF},\mathrm{loc}}(X)\cap C(X), |d f |\leq 1 \}.
\end{align*}
\end{theorem}

\begin{proof} \textbf{(i)} See \cite[Proposition 5.1]{eellsHarmonicMapsRiemannian2001}. \\
\textbf{(ii)} See \cite[Corollary 4.1]{eellsHarmonicMapsRiemannian2001}. \\
\textbf{(iii)} See \cite[Theorem 5.1]{eellsHarmonicMapsRiemannian2001}. \\
\textbf{(iv)} The first equality is proved in \cite[Proposition 4.1]{eellsHarmonicMapsRiemannian2001}. Let us demonstrate the second. We define
\[
\rho(x,y) := \sup\{ f(x)-f(y) \mid f\in W^{1,2}_{\mathrm{EF},\loc}(X)\cap C(X), |d f |\leq 1 \}.
\]
In view of 
$$
\LIP(X)\subset W_{\mathrm{EF},\loc}^{1,2}(X)\cap C(X)
$$
we immediately get $d\leq\rho$. For the reverse inequality, we pick $f\in W^{1,2}_{\mathrm{EF},\loc}(X)\cap C(X)$ such that $|d f |\leq 1$. Let $T=(K,\vartheta)$ be a locally Lipschitz triangulation, let $s\in\mathcal{S}^{(m)}(X,T)$ and let $\phi$ be an affine coordinate system for $s$. The map $f\circ \vartheta\circ \phi^{-1}$ is in the Euclidean Sobolev space $W^{1,\infty}(\phi(s^\circ))\subset \LIP(\phi(s^\circ))$, and so $f \in \LIP_\loc(X)$. Being equipped with this regularity, with the notation as above, set
$$
Z:=Z_T\cup \{|d f |>1\}.
$$
For every $\gamma\in\Omega(x,y;Z)$ we have
$$
(d/dt) f(\gamma(t))=\sum^m_{i=1}\partial_i (f\circ \vartheta \circ \phi^{-1}) (\gamma^1(t),\dots,\gamma^m(t)) \dot{\gamma}^i(t),
$$
for a.e. $t\in \gamma^{-1}(\vartheta(s^\circ))$, which using Cauchy-Schwarz and $|d f  |(\gamma(t))\leq 1$ for a.e. $t\in [0,1]$ straightforwardly implies
$$
|f(y)-f(x)|\leq \int^1_0 |(d/dt) f(\gamma(t))| \dd t\leq \mathscr{L}(\gamma),
$$
showing that $f$ is $1$-Lipschitz on $(X,d)$, and so $d\geq\rho$.
\end{proof}

\begin{theorem}
\th\label{riemPolyhedra:th:EqTwoDefSobolev} Assume $d$ is complete and there exists a locally Lipschitz triangulation $T=(K,\vartheta)$ such that for all $s\in\mathcal{S}^{(m)}(X,T)$ there exists an affine coordinate system $\phi$ for $s$ and the map $g^T_s\circ\phi^{-1}$ is smooth (in short: a \emph{polyhedrally smooth Riemannian metric}). Then one has 
$$
W_{\mathrm{EF}}^{1,2}(X) = W_{\mathrm{AGS}}^{1,2}(X, d, \mu),
$$
and for every $f\in W_{\mathrm{AGS}}^{1,2}(X,d,\mu)$, 
\[
|Df| = |df|\quad\text{$\mu$-a.e. on}~X.
\]
\end{theorem}

\begin{proof} As $g^T_s\circ\phi^{-1}$ is smooth for every $s$, one has
\begin{equation}
\label{riemPolyhedra:eq:localLipEqGrad}
\lip(f)(x) = |df(x)| \quad\text{for $\mu$-a.e.}~x\in X.
\end{equation}
We define the functional
\[
\begin{aligned}
\mathcal{E} : L^2(X) &\longrightarrow   [0,\infty], \\
f &\longmapsto
\begin{cases}
\int |df|^2\dd\mu, & f\in W_{\text{EF}}^{1,2}(X), \\
\infty, & f\notin W_{\text{EF}}^{1,2}(X).
\end{cases}
\end{aligned}
\]
Since $W_{\text{EF}}^{1,2}(X)$ is a Hilbert Space, it is clear that $\mathcal{E}$ is lower semicontinuous. Let $f\in W_{\text{EF}}^{1,2}(X)$. By definition, there exists a sequence $(f_k)$ in $\LIP^{1,2}(X)$ that converges to $f$ in the $W^{1,2}_{\mathrm{EF}}$-norm. Thus we have
\[
\mathcal{E}(f) = \lim\mathcal{E}(f_k) = \lim\int\lip(f_k)^2\dd\mu \geq \Ch(f),
\]
by the definition of $|Df|$, and
\[
\mathcal{E}(f) \leq \liminf\mathcal{E}(f_k) = \liminf\int\lip(f_k)^2\dd\mu,
\]
and so $\mathcal{E}(f)\leq\Ch(f)$. We conclude that $W_{\mathrm{EF}}^{1,2}(X,d,\mu)\subset W_{\mathrm{AGS}}^{1,2}(X,d,\mu)$ and 
\[
\mathcal{E}(f) = \Ch(f) \quad\text{for all}~f\in W_{\text{EF}}^{1,2}(X).
\]
For the reverse inclusion, let $f\in W_{\mathrm{AGS}}^{1,2}(X,d,\mu)$.
As $\Ch(f)<\infty$, by definition, there exists a sequence $(f_k)$ in $\LIP(X)\cap L^2(X)$ such that $f_k\to f$ in $L^2(X)$, and so
\[
\mathcal{E}(f) \leq \liminf_{k\to\infty}\mathcal{E}(f_k) = \liminf_{k\to\infty}\int\lip(f_k)^2\dd\mu.
\]
By taking the infimum over all such sequences, we obtain $\mathcal{E}(f)\leq\Ch(f)$.
\end{proof}

Note that in the above situation, as the set of non-Lebesgue points of the polyhedral Riemannian metric is empty, one gets the simple formula
$$
d(x,y)=\inf \mathscr{L}(\gamma), 
$$
where the infimum is over all locally Lipschitz paths from $x$ to $y$·

\begin{remark}\label{rypo} Under the assumptions of Theorem \ref{riemPolyhedra:th:EqTwoDefSobolev}, the conclusions of \thref{riemPolyhedra:th:analyticProps} remain true if we replace $W_{\text{EF}}^{1,2}(X)$
by $W_{\text{AGS}}^{1,2}(X,d,\mu)$ and $|d f |$ by $|Df|$. 
\end{remark}

One may wonder if it is possible to drop the smoothness assumptions on the Riemannian metric in Remark \ref{rypo}. However, in \cite[Section 4.1]{ryborzInfinitesimalStructure2025}, the author has constructed a measurable Riemannian metric $g$ that satisfies a uniform ellipticity condition and a $C^1$-function $f$ with 
$$
\min(|Df|,\lip(f))>|\nabla f|
$$
on a set of positive measure.

\bibliography{bibliography}

\end{document}